\documentclass[a4paper, 12pt]{article}
\usepackage[utf8]{inputenc}
\usepackage[T2A]{fontenc}
\usepackage{amsmath,amssymb,amsthm}
\usepackage[a4paper,hmargin=2.5cm,vmargin=2.5cm]{geometry}
\usepackage[english]{babel}
\usepackage[all]{xy}
\usepackage{xcolor}
\usepackage[normalem]{ulem}
\usepackage{color}

\title{On Bruhat intervals of small lengths for Weyl groups}
\author{Eugenia Akhmedova}

\begin{document}
\maketitle

\newtheorem{conj}{Conjecture}
\theoremstyle{definition}
\newtheorem{lemma}{Lemma}

\begin{abstract}
The paper considers the Bruhat order on Coxeter groups in general and Weyl groups in particular.We suggest an invariant on the Bruhat order intervals that allows checking whether two Bruhat intervals are isomorphic. This invariant is constructed inductively and allows us to compute short length intervals for Coxeter groups of types A and B. The results are presented in the appendices at the end of the paper. 
\end{abstract}

\section*{Introduction}
The number of Bruhat intervals in Coxeter groups is finite, and for the first few lengths, the intervals were described up to an isomorphism by A. Hultman in \cite{Hultman} and \cite{Hultman2} using the correspondence between Bruhat intervals and cell decompositions of a 2d sphere and straightforward computations.
The main purpose of this paper consists of a description of the intervals in higher dimensions, as the Hultman's geometric method  is hard to apply due to rapidly growing with length number of nonisomorphic intervals.
We construct an invariant on subintervals in the Bruhat graphs, using their specific properties. This gives us a method of comparing two Bruhat interval, that is faster than the general algorithm for checking if two graphs are isomorphic. This construction is inductive, and thus, can be easily applied for any interval length and Weyl group.

In section \ref{basics}, we recall some basic definitions and properties concerning Coxeter and Weyl groups and the Bruhat order. Subsection \ref{overview} gives a brief overview of previous results in Bruhat interval classification.
In section \ref{results}, we construct a combinatorial invariant on Bruhat intervals and show how it can be used to classify the intervals.
We sketch the algorithm of the inductive classification of Bruhat intervals of given length based on the invariant we suggest.\footnote{https://github.com/GeanAkhmedova/Bruhat-intervals.git} At the very end of the paper we point out a curious observation on the computational results. It turns out that for $n$ up to 6 any two Bruhat intervals of length $k$ in $A_n$ are isomorphic if and only if there exists a collection of isomorphisms between their $k-1$ subintervals.

Appendices A and B provide a list of representatives of equivalency classes with respect to poset isomorphisms for intervals of lengths 5 and 6 in Weyl groups of type A (25 and 103 intervals respectively) and intervals of lengths 4 and 5 in Weyl groups of type B (17 and 143 intervals respectively).

\newtheorem{definition}{Definition}
\newtheorem{theorem}{Theorem}
\newtheorem{corollary}{Corollary}
\newtheorem{stat}{Proposition}

\section{Preliminaries}\label{basics}
In this section we recall some basic facts about Coxeter and Weyl groups and the Bruhat order on them. The proof of all the propositions below can be found in the first two sections of the book by Bjorner and Brenti \cite{Bjorner}. 
\subsection{Coxeter groups}

\begin{definition}
A Coxeter system is a pair $(W,S)$ made up of a group $W$ and its generating set $S$ such that $\forall s  \in S \ s^2 = e $ and all relations in $W$ are generated by the relations of the type $(s_1s_2)^{m(s_1,s_2)} = e$,$s_1,s_2\in S$.\\
\end{definition}

\begin{definition}
 Each element $w\in W$ of the Coxeter group can be related to two sets $T_L(w)$ and $T_R(w)$. The sets $T_L = \{t \in T \ | \ \ell(tw) < \ell(w)  \}$  and $T_R = \{t \in T \ | \ \ell(wt) < \ell(w)\}$ are called the left and right sets of associated reflections. The sets $D_L = T_L \cap S$ and $D_R = T_R \cap S$ are called sets of left and right descents, respectively.    
\end{definition}

For any $w \in W$ there exists an expression $w = s_1s_2\dots s_k, \ s_i\in S$. The minimal length of such expression is called the length of the element $w$ and is denoted by $\ell(v)$. If $s_1s_2\dots s_k$ is of minimal length, it is called a reduced word for $w$.

\begin{definition}
A reflection set of the Coxeter system $(W,S)$ is the set $T = \{wsw^{-1}|s\in S,w \in W \}$.
\end{definition}

\begin{stat}
\textbf{Strong exchange property}\\
Suppose $w \in W$, $w = s_1s_2\dots s_k$, and $t \in T$. If $\ell(tw) < \ell(w)$, then $tw = s_1s_2\dots \hat{s_i} \dots s_k$, where $\hat{s_i}$ denotes removing $s_i$ from the word.
\end{stat}
\begin{stat}
\textbf{Deletion property}\\
Suppose $w \in W$, $w = s_1s_2\dots s_k$ and $\ell(w) < k$. Then there exist $1 \leq i < j \leq k$ such that $w = s_1s_2\dots \hat{s_i} \dots \hat{s_j} \dots s_k$.
\end{stat}

\begin{theorem} \emph{(proven by Dyer in \cite{Dyer2})}

Suppose $(W,S)$ is a Coxeter system, $W'$ is a subgroup in $W$, generated by a subset of the set of reflections $T$.
Let $S' = \{t\in T| \ell(t't) > \ell(t) \ \forall t' \in T\}$.
Then $(W',S')$ is a Coxeter system.
\end{theorem}

\begin{stat}
The number of nonisomorphic Coxeter systems with $n$ generators is finite.
\end{stat}

\subsection{Bruhat order}
Let us now introduce an order on a Coxeter group $W$.
\begin{definition}
Let $u,w \in W$. Then:
\begin{align*} 
&(1) \ u \rightarrow w \text{ means } \ell(u) < \ell(v) \text{ and } \exists \ t \in T\text{ means }w = ut,\\
&(2) \ u \lhd w \text{ means } u \rightarrow w \text{ and } \ell(w) = \ell(u) + 1, \\
&(3) \ u \leq w \text{ means there exists a chain in }W,\\
& \ \ \ \ \text{ such that } u = u_0 \rightarrow u_1 \rightarrow \dots  \rightarrow u_{k-1} \rightarrow u_k = w.
\end{align*}
\end{definition}

\begin{definition}
A Bruhat graph is a graph with vertices labeled by elements of $W$ with edges $(u,w)$ such that $u \rightarrow w$.
\end{definition}

\begin{definition}
A subword of a reduced word $s_1s_2\dots s_q$ is a word $s_{i_1}s_{i_2}\dots s_{i_k}$ such that $1\leq i_1 < i_2 < \dots < i_k \leq  q$
\end{definition}

\begin{lemma}
Suppose $u,w \in W$,$u \neq w$, $w = s_1s_2\dots s_q$ is a reduced word, and there exists a reduced word for $u$ that is a subword of $s_1s_2\dots s_q$. Then  there exists $v \in W$ that satisfies the following conditions:
\begin{align*} 
&(1) \ v > u\\
&(2) \ \ell(v) = \ell(u) + 1 \\
&(3) \ \text{There exists a reduced word for } v, \text{ that is a subbword of } s_1s_2\dots s_q
\end{align*}
\end{lemma}

\begin{corollary}

\textbf{Subword property}\\
Suppose $w = s_1s_2\dots s_q$ is a reduced word. Then $u \leq w$ if and only if there exist a reduced word for $u$ that is a subword of $s_1s_2\dots s_q$.
\end{corollary}

\begin{stat}
If $W$ is finite, then there exists an element of maximal length $w_0 \in W$.
\end{stat}

\begin{stat}
The element of maximal length $w_0$ is the maximal element with respect to the Bruhat order.
\end{stat}

\begin{corollary} 
The maps $w \mapsto ww_0$ and $w \mapsto w_0w$ are any automorphisms of the Bruhat order.
\end{corollary}

\begin{stat}\textbf{Chain property}\\
Suppose $u < w$. Then there exists a chain $u = x_0 < x_1 < \dots < x_k = w$ such that $\forall i$ $\ell(x_{i+1}) = \ell(u) + i$.
\end{stat}

\subsection{Bruhat interval structure}

\begin{definition}
A Bruhat interval $[u,w]$ is the set $\{v \in W \ | \ u \leq v \leq w \}$. The length of a Bruhat interval $[u,w]$ is $\ell([u,w])= \ell(w) - \ell(u)$.
The set of all chains of the maximal length in the interval $[u,w]$ is denoted by $\mathcal{M}(u,w)$.
\end{definition}

The total order on $\mathcal{M}(u,w)$ is introduced by associating each $m \in  \mathcal{M}(u,w)$ with an ordered set of natural numbers $\lambda(m)$ as will be described below and ordering those sets and corresponding elements of $\mathcal{M}(u,w)$ lexicographically. 
Suppose $m$ is a chain $w = x_0 \rhd x_1 \rhd \dots \rhd x_n = u$. Let us fix a reduced subword $w = s_1s_2\dots s_q$. From the strong exchange property, it follows that there exists a unique $i$ such that $x_1 = s_1s_2\dots \hat{s_i} \dots s_q$.

Let us fix $\lambda_1(m) = i$. The following $\lambda(m)$ are constructed through repeating the same process. If after $k$ repetitions, the last deleted letter was $s_j$, then one can fix $\lambda_k(m) = j$. The order $\prec$ on $\mathcal{M}(u,w)$ is introduced as a lexicographic order on $\lambda(m)$. 

\begin{stat}
 There exists unique a chain $m_0 \in \mathcal{M}(u,w)$  such that $ \lambda(m_0) $ is increasing.
\end{stat}

\begin{stat}
Such $m_0 \in \mathcal{M}(u,w)$ that $\lambda(m_0)$ is increasing is the least element in $\mathcal{M}(u,w).$
\end{stat}

\begin{corollary} \label{diam}
Any length 2 Bruhat interval is isomorphic to $\vcenter{\xymatrix @-1pc {
& {\circ}& \\
{\circ}\ar@{-}[ur] & & {\circ} \ar@{-}[ul] \\
& {\circ} \ar@{-}[ur] \ar@{-}[ul] &}}$.
\end{corollary}

\begin{corollary} \label{shellable}
Let $m, m' \in \mathcal{M}(u,w)$, $m' \prec m$. Then there exists $k \in \mathcal{M}(u,w)$ such that $k \prec m$, $m' \cap m \subset k \cap m$ and $|k\cap m| = |m| - 1$.
\end{corollary}

\begin{theorem} \label{subinterval}
Let $(W,S)$ be a finite Coxeter system and $(u,v)$ - an open Bruhat interval in $W$ of length $n$. Then there exists a reflection subgroup $(W',S')$ in $(W,S)$ and a Bruhat interval $(u',v')$ in $(W',S')$ such that $|S'| \leq n$ and $(u',v')$ is isomorphic to $(u,v)$. 
\end{theorem}

\begin{corollary}
The number of distinct with respect to isomorphisms Bruhat intervals of length $n$ in finite Coxeter groups is finite.
\end{corollary}

\subsection{Weyl groups}
For any vector $\alpha \in \mathbb{E}^d$, let us denote by $\sigma_{\alpha}$ the reflection against the hyperplane orthogonal to $\alpha$.
\begin{definition}
A crystallographic root system is a finite set $\Phi \subset \mathbb{E}^d \ \{0\}$ such that $\Phi$ spans $\mathbb{E}^d$ and for any $\alpha,\beta \in Phi$, the following conditions hold:
\begin{align*}
(1)& \ \Phi \cap \mathbb{R}\alpha = \{\alpha, -\alpha\},\\
(2)& \ \sigma_{\alpha}(\Phi) = \Phi, \\
(3)& \ \sigma_{\alpha}(\beta) = \beta + n\alpha, \ n \in \mathbb{Z}.
\end{align*}
\end{definition}

\begin{stat}
Any Weyl group is a Coxeter group.
\end{stat}

\begin{stat}  \label{weylsub}
A subgroup of a Weyl group $W$, generated by a subset of the reflection set $T$, is also a Weyl group.
\end{stat}

\begin{stat}  \label{weylclass}
Weyl groups correspond to Coxeter diagrams of the types  $A_n$, $B_n$, $D_n$, $E_6$, $E_7$, $E_8$, and $I_2(6)$.
\end{stat}

\subsection{Explicit description of short intervals in Weyl groups} \label{overview}
Hultman describes Bruhat intervals of length 4 up to an isomorphism for all Weyl groups and Bruhat intervals of length 5 for the symmetric groups in~\cite{Hultman,Hultman2}. In order to describe intervals of length 4 Hultman uses the relation between the order complex of a Bruhat interval and cell decompositions of a sphere. Straightforward computations in Maple are used to check this result, as well as obtain the result for intervals of length 5 in the symmetric groups.

 \section{Invariant on Bruhat intervals} \label{results}
 
 In this section, we construct a correspondence between Bruhat intervals and special arrays consisting of graph isomorphisms and natural numbers. Let us first give a short informal description of the proposed construction:
 
 The construction is inductive. Suppose that we have classified Bruhat intervals in $W$ of length $n-1$ and enumerated the equivalency classes. Then for an interval of length $n$, one can find a number of equivalency classes of its length $n-1$ subintervals with the same smallest element as the bigger interval.
 From these numbers and some additional information, one can "glue together" representatives of the respective equivalency classes and get an interval, isomorphic to the original. So the invariant for each Bruhat interval $[u,v]$ of the length $n$ consists of a set of Bruhat intervals of length $n-1$ and an instruction how to make an interval, isomorphic to $[u,v]$ out of them.
 This set of numbers and the additional information that will be described below together make up the suggested invariant. 

It turns out that the value of the suggested invariant is the same for all isomorphic Bruhat intervals, which, knowing all length $n-1$ intervals up to an isomorhism, allows us to describe length $n$ Bruhat intervals up to an isomorhism.

\subsection{Invariant construction}

Suppose that for every $ k < n $, a set  $I_k$ was constructed, consisting of representatives from isomorphism equivalency classes of Bruhat intervals of length $k$, one from each. Suppose also that for each interval of the length $ k < n $ in $W$, an isomorphism $\varphi_{[u,v]}$ is fixed into one of the intervals from $I_k$.

Let us introduce three maps $T$, $P$ and $A$ on the set on intervals of length $n$ such that the image of $[u,v]$ is defined as follows:

Let us label the elements of the set $\{ w \in W \ | u < \ w \lhd v \}$ as $v_1,\dots,v_n$ and fix the set of automorphisms $\psi = (\psi_1, \dots, \psi_n)$, where $\psi_i$ is an automorphism of  the interval $[u,v_i]$. Denote the maps $\varphi_{[u,v_i]}$ as $\varphi_i$ and the maps $\varphi_{[u,v_i]} \circ \psi_i$ as $\Theta_i$. 

Now let us define the image of $T$ as the set of numbers of intervals in $I_{n-1}$ isomorphic to $[(u,v_i)]$, repetitions included.

Define the map $P_{\psi}$ as taking an interval $[u,v]$ to the set $\{ (i,j,\Theta_i(w),\Theta_j(w))| 1 \leq i<  j \leq n, w \lhd v_i,v_j\}.$

Denote $\tau_{(u,v),w}^{\psi} = \varphi_{[\Theta_i(u), \Theta_i(w)]} \circ
\Theta_i | _{[u,w]} \circ \Theta_j^{-1}| _{\Theta_j([u,w])} \circ \varphi^{-1}_{[\Theta_j(u),\Theta_j(w)])} $.

The map $\tau_{(u,v),w}^{\psi}$ is well-defined, because, as follows from Corollary \ref{diam}, $v_i$ and $v_j$ are uniquely defined up to a permutation by $w$. Note that $\tau_{(u,v),w}^{\psi}$ is an automorphism of the interval $\varphi_{[u,w]}([u,w])$. 

Define the map $A_{\psi}$ as taking an interval $[u,v]$ into the set $\{(i,\Theta_i(w), \tau_{(u,v),w}^{\psi}
| 1 \leq i \leq n, w \lhd v_i\}.
$

Let us unite the constructed maps together into the map 
 $TPA$ such that  $$TPA([u,v]) = \{(T([u,v]),P_{\psi}([u,v]),A_{\psi}([u,v]))|\psi \text{ - set of automorphisms}, \ \psi_i \in Aut([u,v_i]) \}.$$ 
 Note that the $TPA$ map does not depend on the choice of $\psi_i$, but does depend on the way we enumerated the elements of the set $\{ w \in W \ | u < \ w \lhd v \}$.

To illustrate this construction, let us give an example.
The picture below presents a commutative diagram, showing the relation between $\Theta_i$ and $\tau_{(u,v),w}^{\psi}$ for a particular interval of length 3. To make the diagram more readable, the edges of the graph are subscribed in black and the maps are subscribed in blue. The subinterval $[u,w]$ and its images and preimages are marked by filled-in points for the vertices and solid lines for the edges.

$
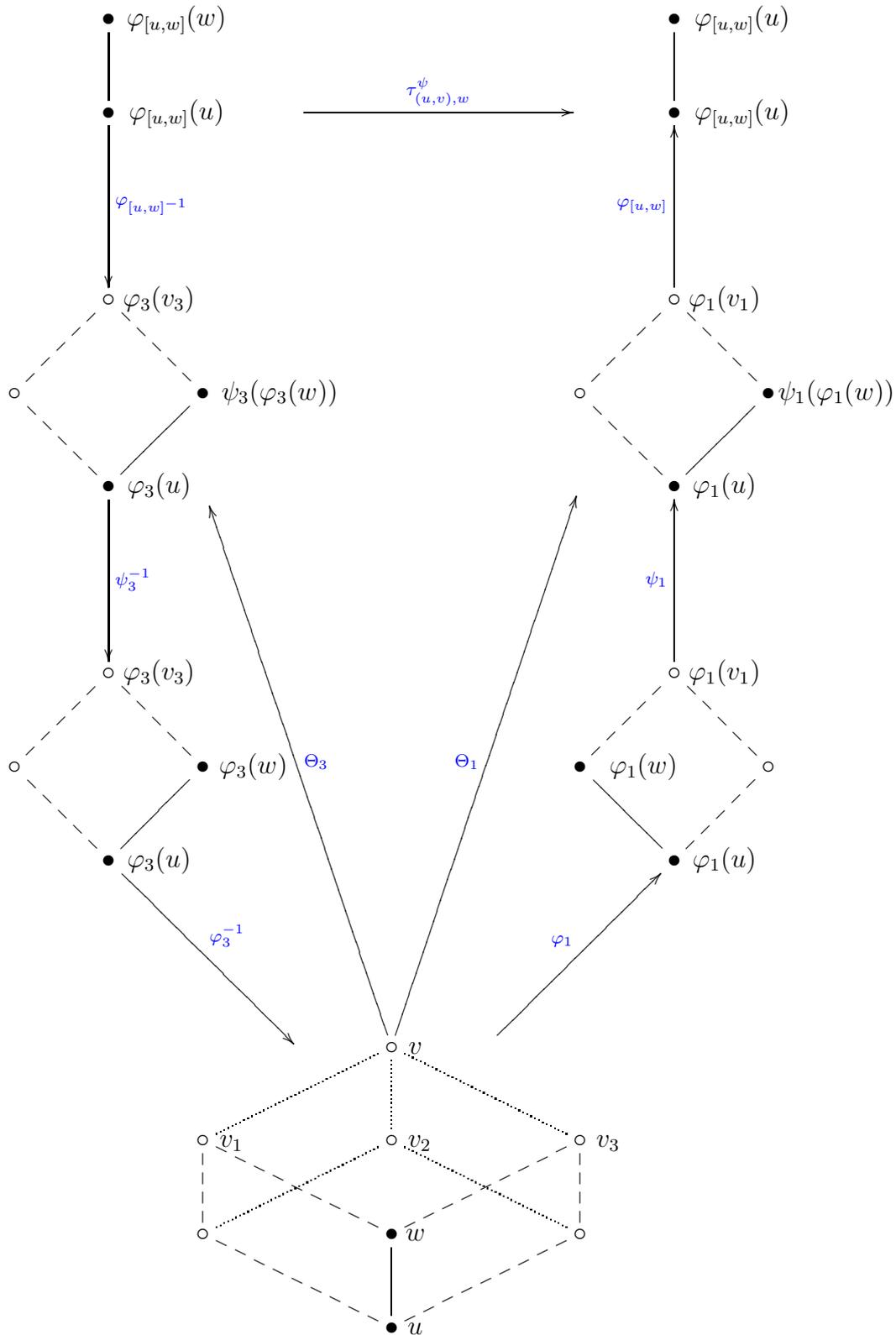
\xymatrix @=3.5pc @!0 {
&{\bullet} \ar@{-}[d] & {\varphi_{[u,w]}(w) \ \ \ \ \ \ }  &&&&& {\bullet}& {\varphi_{[u,w]}(u) \ \ \ \ \ \ }\\
&{\bullet} \ar[dd]^{\color{blue}\varphi_{[u,w]^{-1}}} & {\varphi_{[u,w]}(u) \ \ \ \ \ \ } & \ar[rrr]^{\color{blue}\tau_{(u,v),w}^{\psi}}&&&& {\bullet} \ar@{-}[u]& {\varphi_{[u,w]}(u) \ \ \ \ \ \ }\\
\\
& {\circ} & {\varphi_3(v_3) \ \ \ \ \ \ \ \ \ \ } &&&&& {\circ} \ar@{--}[dr] \ar[uu]^{\color{blue}\varphi_{[u,w]}} & {\varphi_1(v_1)  \ \ \ \ \ \ \ \ \ \ }\\
{\circ}\ar@{--}[ur] & & {\bullet} \ar@{--}[ul] & {\psi_3(\varphi_3(w))\ \ \ \ } &&&{\circ} \ar@{--}[ur]&&{\bullet} \ar@{-}[dl]&{\psi_1(\varphi_1(w)) \ \ \ \ \ \ }\\
& {\bullet} \ar@{-}[ur] \ar@{--}[ul] \ar[dd]^{ \color{blue} \psi_3^{-1}}& {\varphi_3(u)  \ \ \ \ \ \ \ \ \ \ } &&&&&{\bullet} \ar@{--}[ul]&{\varphi_1(u)  \ \ \ \ \ \ \ \ \ \ }\\
\\
& {\circ} & {\varphi_3(v_3) \ \ \ \ \ \ \ \ \ \ } &&&&& {\circ} \ar@{--}[dr] \ar@{--}[dl] \ar[uu]^{\color{blue}\psi_1}&{\varphi_1(v_1)  \ \ \ \ \ \ \ \ \ \ }\\
{\circ}\ar@{--}[ur] & & {\bullet} \ar@{--}[ul]& {\varphi_3(w)  \ \ \ \ \ \ \ \ \ \ } &&&{\bullet} \ar@{-}[dr] &{\varphi_1(w) \ \ \ \ \ \ \ }& {\circ} \ar@{--}[dl]\\
& {\bullet} \ar@{-}[ur] \ar@{--}[ul] \ar[ddrr]^{\color{blue}\varphi_3^{-1}}& {\varphi_3(u)  \ \ \ \ \ \ \ \ \ \ } &&&&&{\bullet}&{\varphi_1(u)  \ \ \ \ \ \ \ \ \ \ }
\\
\\
&&&&{\circ} \ar@{.}[drr]\ar@{.}[d] \ar@{.}[dll] \ar[-6,-2] _{\color{blue}\Theta_3} \ar[-6,2]^{\ \ \ \ \color{blue}\Theta_1} & {v \ \ \ \ \ \ \  \ \ \ \ \ \ \ \  \ \ar[uurr]^{\color{blue} \varphi_1} } \\
&&{\circ} \ar@{--}[d] \ar@{--}[drr] &{v_1 \ \ \ \ \  \ \ \ \ \ \ \ \  \ \ }& {\circ} \ar@{.}[drr] \ar@{.}[dll] &{v_2 \ \ \ \ \ \ \  \ \ \ \ \ \ \ \ }& {\circ} \ar@{--}[d] \ar@{--}[dll]& {v_3 \ \ \ \ \ \ \  \ \ \ \ \ \ \ \ }\\
&&{\circ} \ar@{--}[drr] && {\bullet} \ar@{-}[d]& {w \ \ \ \ \ \  \ \ \ \ \ \ \ \  \ \  } & {\circ} \ar@{--}[dll] \\
&&&&{\bullet} & {u \ \ \ \ \ \  \ \ \ \ \ \ \ \  \ \  }}$

Let us come back to the informal explanation of the construction. Now it is possible to say exactly what "gluing together" meant in terms of the newly introduced notation. To get an interval isomorphic to $(u,v)$ of length $n$, one can take intervals from $I_{n-1}$ with numbers from $T([u,v])$ and identify the images of each particular vertex under  
every $\varphi_i$ that acts on it. The equivalency on the vertices of length $\ell(v)-2$ is given by $P_{Id}([u,w])$, and as one knows the "standard" isomorphisms from subintervals of length $n-2$ into elements of $I_{n-2}$, the described equivalency can be reconstructed from $A_{Id}([u,v])$.

Thus, $T([u,v]), P_{Id}([u,v]), A_{Id}([u,v])$ uniquely identify the interval $[u,v]$. However, the next part will consider the more general attribute  $TPA$, as it proves useful for comparing intervals. 

\subsection{Properties of $TPA$}
\begin{lemma}\label{techlem1}
Let $s,u,v \in W$, $s < u < v$, $\{v_1,v_2, \dots ,v_n\} = \{w \in W| s < w \lhd v\}$ and $u < v_i,v_j$. Then there exists a sequence $v_i = v_{i_0}, v_{i_1}, \dots, v_{i_k} = v_j$ such that $\forall q \ \exists r \in W :\ u < r \lhd v_{i_q},v_{i_{q+1}} $.
\end{lemma}

\begin{proof}
Let us fix a reduced word $v = s_1s_2\dots s_q$. Suppose $v_1 = 
s_1s_2\dots \hat{s_{i_1}} \dots s_q$ and $v_2 =  s_1s_2\dots 
\hat{s_{i_2}} \dots s_q$, without loss of generality $i_1 < i_2$.
Consider the chains $u = u_m \lhd u_{m-1} \lhd \dots \lhd
u_{1} = v_1 \lhd u_0 = v$ and $u = \tilde{u_m} \lhd 
\tilde{u_{m-1}} \lhd \dots \lhd \tilde{u_1} = v_2 \lhd \tilde{u_0}
= v $. By Corollary \ref{shellable}, there exists a sequence $C_0,C_1, \dots,C_p$ of chains from $\mathcal{M}(u,w)$ such that $C_0 = (u_0,u_1, \dots, 
u_m)$, $C_p = (\tilde{u_0},\tilde{u_1}, \dots ,\tilde{u_m}$,  
$\forall i$ $\lambda(C_i) < \lambda(C_{i+1})$ and $\forall i 
|C_i\cap C_{i+1}| = m$. 

Then the chain $v_1 = (C_0)_1,(C_1)_1, \dots, (C_{p-1})_1,
(C_p)_1 = v_2$ satisfies the required condition, because if $(C_i)_1 \neq (C_{i+1})_1$ then  $(C_i)_2 = 
(C_{i+1})_2 \lhd (C_i)_1,(C_{i+1})_1$.
\end{proof}

\begin{theorem}
If $TPA([u,v]) = TPA([u',v'])$ then $[u,v] \cong [u',v']$.
\end{theorem}

\begin{proof}
Suppose that $u,u',v,v' \in W$, $TPA([u,v]) = TPA([u',v'])$. Then there exists $\psi$ such that $P_{\psi}([u,v]) = P_{Id}([u',v'])$ and $A_{\psi}([u,v]) = A_{Id}([u',v'])$. Denote the maps $\varphi_{[u',v_i']}$ as $\Theta_i'$. 
It is necessary to show that the map $f: [u,v] \to [u',v']$ such that $\forall i \ f: w \to \Theta_i'^{-1}(\Theta_i(w))$ is a Bruhat order isomorphism.
Suppose $u < w_1 < w_2 < v$. Then $\exists i: w_2 \leq v_i$ and $f(w_1) < f(w_2)$ follow from the fact that $\Theta_i'^{-1}\circ \Theta_i$ is a Bruhat order isomorphism.
Now the only thing left to prove is that $f$ is defined correctly. This is equivalent to the statement that any $i,j$ and for any $x < v_i, v_j$ satisfies  $\Theta_i'^{-1}(\Theta_i(x)) =  \Theta_j'^{-1}(\Theta_j(x))$.
Due to Lemma \ref{techlem1}, it is enough to prove this proposition for $w \lhd v_i,v_j$.

From $T([u,v]) = T([u',v'])$ follows $[u,v_i] \cong [u',v_i']$. Therefore,$\Theta_i(u) = \Theta_i'(u')$, $\Theta_i(v_i) = \Theta_i'(v_i')$. From 
$P_{\psi}([u,v]) = P_{Id}([u',v'])$ follows $\exists w' \lhd 
v'_i,v'_j$ such that $\varphi_i(w)= \Theta_i'(w')$ and
$\Theta_j(w)= \Theta_j'(w')$.

Then from $A_{\psi}([u,v]) = 
A_{Id}([u',v'])$ follows $\tau_{(u,v),w}^{\psi} = 
\tau_{(u',v'),w'}^{Id}$. Therefore, 
$$\varphi^{-1}_{[\Theta_i(u),\Theta_i(v_i)]}\tau_{(u,v),w}^{\psi}
\varphi_{[\Theta_j(u),\Theta_j(v_i)]} = 
\varphi^{-1}_{[\Theta_i'(u'),\Theta_i'(v_i')]}\tau_{(u',v'),w'}^
{Id}\varphi_{[\Theta_j'(u'),\Theta_j'(v_i')]},$$

This means that
$\Theta_i | _{[u,w]} \circ \Theta_j^{-1}| _{\Theta_j([u,w])} = 
\Theta_i' | _{[u',w']} \circ \Theta_j'^{-1}| 
_{\Theta_j'([u',w'])}$.

By multiplying by  $\Theta_i | 
_{[u,w]}^{-1}$  on the left and $\Theta_j' | _{[u',w']}$ on the right, one can get 
$\Theta_j'^{-1}\circ \Theta_j|_{[u,w]} = \Theta_i'^{-1}\circ 
\Theta_i|_{[u,w]}$.
\end{proof}

\begin{theorem}\label{intiso}
Suppose $[u,v] \cong [u',v']$. Then there exists an order of elements in $\{w \in W| u < w \lhd v\}$, and $\{w \in W| u' < w \lhd v'\}$ such that $TPA([u,v]) = TPA([u',v'])$.
\end{theorem}

\begin{proof}
 Suppose that $f:[u,v] \to [u',v']$ is an isomorphism of Bruhat order and $\{v_1,v_2, \dots, v_n\}$ is an enumeration of elements of $\{w \in W| u < w \lhd v\}$. 
 Let us consider an enumeration of elements of $\{w \in W| u' < w \lhd v'\}$ such that for all $i$'s $v_i' = f(v_i)$. Denote the functions $\varphi_{[u',v'_i]}$ by $\varphi_i'$. Then $[u,v_i] \cong [f(u),f(v_i)] = [u',v'_i]$, and, therefore, $T([u,v]) = T([u',v'])$. 
 
 Since $w \lhd v_i,v_j$, we have $f(w) \lhd v_i',v_j'$, therefore, $P([u,v]) = P([u',v'])$. Because $f$ is a Bruhat order isomorphism, there exists a set of automorphisms $\psi$ such that $\varphi_i'\circ \psi_i = \varphi_i\circ f^{-1}$. Therefore, $\varphi_j' \circ \psi_j\circ(\varphi_i' \circ \psi_i)^{-1} = \varphi_i \circ \varphi_j^{-1}$, in other notation that is $\tau_{(u,v),w}^{Id} = \tau_{(u',v'),f(w)}^{\psi}$. Consequently $A([u,v]) = A_{\psi}([u',v'])$. This gives the required equality $TPA([u,v]) = TPA([u',v'])$.
\end{proof}

\subsection{Applying the construction to computations}

Analogically to the construction of $TPA$, let us build a program for finding all nonisomorphic Bruhat order intervals in some particular Weyl group. The maps $\varphi_{[u,v]}$ are constructed from corresponding maps for smaller length intervals in the same way they are in the proof of Theorem $\ref{intiso}$.

To describe all the Bruhat intervals of length $n$ in the Weyl group $W$, let us first construct $\varphi_{[u,v]}$ for every interval $[u,v]$ for all intervals of length $n-1$ such that $\ell(v) \leq \ell(w_0) - \frac{n}{2}$. Then let us find $TPA_{Id}([u,v])$, and, if it is not in $TPA$ for any interval already in $I_n$, let us add both $[u,v]$ and $[vw_0,uw_0]$ into $I_n$ and find $TPA$ for them to be compared to subsequent intervals. It is sufficient to only go through the intervals starting in the lower half, since $w \mapsto ww_0$ is an antiautomorphism of Bruhat order on $W$.

If for some $[u',v'] \in I_n$ the set $TPA_{Id}([u,v])$ is contained in  $TPA([u',v'])$, i.e. there exist such $\psi$ that $TPA_{Id}([u,v]) = TPA_{\psi}([u',v'])$, then from $\varphi_{[u,v_i]}, \varphi_{[u',v'_i]}$ and $\psi$, one construct  $\varphi_{[u,v]}$.

From Proposition \ref{weylclass} and Theorem \ref{subinterval}, it follows that listing all Bruhat intervals of length $n$ is comprised of listing all Bruhat intervals in the Weyl groups $A_n$, $B_n$, $D_n$,  $E_6$, $E_7$, $E_8$, and $I_2(6)$. For $n<9$, the groups of type $E$ with less then $n$ generators can be excluded from this list.

The computational results obtained using the above-described algorithm for $n = 3$ and $4$ in type A coincide with the results, presented in \cite{Hultman} and \cite{Hultman2}. The computation time was significantly shorter than stated in those papers, which proves the usefulness of the suggested construction.

Appendices A and B list the representatives of equivalency classes of Bruhat intervals, i.e. elements of the sets $I_n$, for Weyl groups of type $A$, $n = 5,6$ and type $B$, $n = 4,5$. The intervals are denoted by the reduced words for their ends, where $s_i$ are canonically enumerated generating reflections.

\section{Computational observation}
During the computation of the results given in the appendices it was noted that for Weyl groups of type $A$ the $T$ part of the $TPA$ invariant is enough to distinguish non-isomorphic subintervals. The same is not true for Weyl groups of type $B$, with counterexamples existing already in $B_4$. Thus we formulate a conjecture, which as of yet is only supported by computational evidence for $n$ up to 6.

\begin{conj}
Let $u,v,u',v' \in A_n$, $u < v$ , $u'<v'$, $\{v_1,v_2, \dots ,v_n\} = \{w \in A_n| u < w \lhd v\}$, $\{v_1',v_2', \dots ,v_n'\} = \{w \in A_n| u' < w \lhd v'\}$. Suppose $[u,v_i]\cong [u',v_i']$ for all $i$, then $[u,v] \cong [u',v']$.
\end{conj}
A weaker version of this conjecture can be restated in more geometrical terms. 
\begin{conj}
A Bruhat cell in $\mathcal{F}\ell_n$ is determined up to a cell complex isomorphism by the set of cells composing its boundary.
\end{conj}

\pagebreak
\addcontentsline{toc}{section}{Appendix A}
\section*{Appendix A}
\subsection*{Length 5}
25 distinct intervals

$(s_5s_6s_2s_3s_2s_1, s_5s_6s_2s_3s_4s_5s_4s_2s_3s_2s_1),\\
 (s_3s_4s_1s_2, s_2s_3s_4s_5s_2s_3s_4s_1s_2),\\
 (s_4s_5s_6s_4s_2, s_2s_3s_4s_5s_6s_3s_4s_3s_1s_2),\\
 (s_6s_3s_4s_1s_2, s_6s_3s_4s_5s_2s_3s_4s_3s_1s_2),\\
 (s_6s_3s_4s_3s_1s_2, s_4s_5s_6s_1s_2s_3s_4s_2s_3s_1s_2),\\
 (s_5s_1s_2s_1, s_2s_3s_4s_5s_4s_3s_1s_2s_1),\\
 (s_4s_5s_4s_2, s_4s_5s_6s_2s_3s_4s_5s_3s_4),\\
 (s_4s_5s_6s_4s_1s_2s_1, s_4s_5s_6s_4s_5s_3s_4s_2s_3s_1s_2s_1),\\
 (s_3s_4s_5s_4s_2s_1, s_4s_5s_6s_3s_4s_5s_1s_2s_3s_4s_1),\\
 (s_4s_5s_6s_3s_4s_5s_4, s_1s_2s_3s_4s_5s_6s_2s_3s_4s_5s_4s_1),\\
 (s_3s_4s_5s_6s_4s_5s_1, s_1s_2s_3s_4s_5s_6s_3s_4s_5s_3s_4s_1),\\
 (s_5s_6s_5s_1s_2s_3s_2, s_1s_2s_3s_4s_5s_6s_4s_5s_4s_2s_3s_2),\\
 (s_4s_5s_3s_4s_1s_2s_3, s_4s_5s_6s_3s_4s_5s_2s_3s_4s_1s_2s_3),\\
 (s_5s_2s_3, s_3s_4s_5s_2s_3s_4s_3s_2),\\
 (s_4s_1s_2, s_4s_5s_6s_5s_1s_2s_3s_4),\\
 (s_6s_1s_2s_3s_1, s_5s_6s_1s_2s_3s_4s_5s_4s_1s_2),\\
 (s_5s_3s_2, s_3s_4s_5s_6s_2s_3s_1s_2),\\
 (s_6s_4s_5s_3s_4s_3s_1, s_4s_5s_6s_3s_4s_5s_1s_2s_3s_4s_3s_1),\\
 (s_6s_3s_4s_5s_1s_2s_3, s_6s_1s_2s_3s_4s_5s_2s_3s_4s_1s_2s_3),\\
 (s_3s_4s_5s_6s_4s_5s_2, s_3s_4s_5s_6s_2s_3s_4s_5s_3s_1s_2s_1),\\
 (s_6s_5s_3s_4s_3s_1s_2, s_3s_4s_5s_6s_1s_2s_3s_4s_5s_4s_3s_2),\\
 (s_6s_2s_3s_4s_3s_1, s_2s_3s_4s_5s_6s_5s_3s_4s_2s_3s_1),\\
 (s_5s_2s_3s_4s_3s_1, s_1s_2s_3s_4s_5s_1s_2s_3s_4s_2s_3),\\
 (s_6s_3s_4s_5s_3s_4s_1, s_6s_1s_2s_3s_4s_5s_2s_3s_4s_3s_2s_1),\\
 (s_3s_4s_5s_6s_3s_2, s_1s_2s_3s_4s_5s_6s_1s_2s_3s_1s_2)$
\subsection*{Length 6}
103 distinct intervals

$(s_6s_5, s_6s_2s_3s_4s_5s_4s_3s_2),\\
 (s_6s_4s_1s_2s_1, s_6s_4s_5s_1s_2s_3s_4s_2s_3s_1s_2),\\
 (s_5s_3s_4s_1s_2, s_5s_6s_3s_4s_5s_1s_2s_3s_4s_1s_2),\\
 (s_5s_6s_5s_2, s_2s_3s_4s_5s_6s_5s_4s_3s_2s_1),\\
 (s_4s_5s_3s_4s_1, s_4s_5s_6s_3s_4s_5s_1s_2s_3s_4s_1),\\
 (s_3s_4s_5s_1s_2s_3s_2, s_1s_2s_3s_4s_5s_2s_3s_4s_1s_2s_3s_1s_2),\\
 (s_5s_6s_5s_2s_3s_2s_1, s_2s_3s_4s_5s_6s_4s_5s_4s_2s_3s_1s_2s_1),\\
 (s_5s_6s_5s_2s_3s_2, s_2s_3s_4s_5s_6s_5s_1s_2s_3s_4s_3s_2),\\
 (s_4s_5s_6s_2s_3s_2, s_4s_5s_6s_2s_3s_4s_1s_2s_3s_1s_2s_1),\\
 (s_6s_3s_4s_2s_3s_2, s_6s_3s_4s_5s_2s_3s_4s_1s_2s_3s_1s_2),\\
 (s_6s_4s_2s_3, s_4s_5s_6s_2s_3s_4s_5s_1s_2s_3),\\
 (s_6s_5s_2s_3s_2s_1, s_5s_6s_2s_3s_4s_5s_4s_1s_2s_3s_2s_1),\\
 (s_4s_5s_6s_4s_5, s_4s_5s_6s_2s_3s_4s_5s_3s_4s_3s_2),\\
 (s_4s_5s_4s_1s_2s_1, s_6s_2s_3s_4s_5s_3s_4s_1s_2s_3s_2s_1),\\
 (s_3s_4s_5s_3, s_3s_4s_5s_1s_2s_3s_4s_1s_2s_3),\\
 (s_6s_3s_4s_3s_1, s_4s_5s_6s_3s_4s_5s_1s_2s_3s_4s_1),\\
 (s_5s_4s_1s_2s_3s_1s_2, s_1s_2s_3s_4s_5s_6s_4s_5s_3s_4s_1s_2s_1),\\
 (s_5s_6s_4s_5s_4s_2s_1, s_5s_6s_1s_2s_3s_4s_5s_3s_4s_3s_1s_2s_1),\\
 (s_5s_6s_3, s_2s_3s_4s_5s_6s_4s_1s_2s_3),\\
 (s_5s_2s_3s_4, s_4s_5s_6s_4s_5s_2s_3s_4s_3s_1),\\
 (s_4s_5s_6, s_4s_5s_6s_4s_5s_4s_2s_3s_2),\\
 (s_6s_4s_5s_2s_3s_2, s_4s_5s_6s_2s_3s_4s_5s_4s_1s_2s_3s_2),\\
 (s_2s_3s_4, s_1s_2s_3s_4s_5s_1s_2s_3s_4),\\
 (s_2s_3s_4s_5s_3s_4, s_1s_2s_3s_4s_5s_6s_5s_3s_4s_1s_2s_3),\\
 (s_4s_5s_2s_3s_1s_2, s_1s_2s_3s_4s_5s_2s_3s_4s_2s_3s_1s_2),\\
 (s_2s_3s_4s_2s_3s_1, s_3s_4s_5s_6s_2s_3s_4s_5s_1s_2s_3s_4),\\
 (s_3s_4s_5s_6s_5s_4s_2, s_3s_4s_5s_6s_2s_3s_4s_5s_1s_2s_3s_4s_1),\\
 (s_6s_2s_3s_4s_5s_1s_2, s_2s_3s_4s_5s_6s_5s_4s_1s_2s_3s_1s_2s_1),\\
 (s_6s_2s_3s_4s_5s_4s_2, s_3s_4s_5s_6s_2s_3s_4s_5s_3s_4s_3s_1s_2),\\
 (s_6s_4s_3s_1, s_4s_5s_6s_3s_4s_5s_4s_1s_2s_3),\\
 (s_4s_5s_3s_4s_1, s_4s_5s_6s_3s_4s_5s_1s_2s_3s_4s_2),\\
 (s_6s_3s_4s_5s_1s_2s_1, s_3s_4s_5s_6s_1s_2s_3s_4s_5s_4s_3s_2s_1),\\
 (s_6s_2s_3s_4s_5s_4, s_2s_3s_4s_5s_6s_1s_2s_3s_4s_5s_4s_2),\\
 (s_6s_4s_5s_4s_2s_3s_1, s_3s_4s_5s_6s_5s_2s_3s_4s_1s_2s_3s_2s_1),\\
 (s_1s_2s_3s_4s_2s_3s_1, s_1s_2s_3s_4s_5s_6s_5s_2s_3s_4s_1s_2s_3),\\
 (s_2s_3s_4s_5, s_2s_3s_4s_5s_6s_1s_2s_3s_4s_5),\\
 (s_5s_4s_1s_2s_1, s_2s_3s_4s_5s_1s_2s_3s_4s_3s_2s_1),\\
 (s_5s_6s_2s_3s_4, s_2s_3s_4s_5s_6s_4s_5s_1s_2s_3s_4),\\
 (s_5s_6s_4s_1s_2, s_3s_4s_5s_6s_4s_5s_2s_3s_4s_1s_2),\\
 (s_5s_6s_5s_1s_2s_1, s_1s_2s_3s_4s_5s_6s_4s_5s_4s_3s_2s_1),\\
 (s_4s_5s_3s_1s_2s_1, s_1s_2s_3s_4s_5s_3s_4s_1s_2s_3s_1s_2),\\
 (s_4s_5s_3s_4s_2s_3, s_4s_5s_6s_3s_4s_5s_2s_3s_4s_1s_2s_3),\\
 (s_5s_6s_3s_4s_5s_3s_1, s_5s_6s_1s_2s_3s_4s_5s_2s_3s_4s_2s_3s_1),\\
 (s_5s_3s_4s_2, s_5s_6s_3s_4s_5s_2s_3s_4s_1s_2),\\
 (s_6s_1s_2s_3s_4s_3s_2, s_1s_2s_3s_4s_5s_6s_5s_1s_2s_3s_4s_1s_2),\\
 (s_6s_2s_3s_4s_3s_2, s_2s_3s_4s_5s_6s_5s_1s_2s_3s_4s_2s_3),\\
 (s_5s_6s_4s_5s_3s_4, s_5s_6s_4s_5s_3s_4s_1s_2s_3s_1s_2s_1),\\
 (s_5s_6s_3s_4, s_3s_4s_5s_6s_4s_5s_2s_3s_4s_3),\\
 (s_6s_4s_5s_2s_3, s_4s_5s_6s_2s_3s_4s_5s_2s_3s_4s_3),\\
 (s_3s_4s_5s_1, s_3s_4s_5s_6s_2s_3s_4s_5s_1s_2),\\
 (s_4s_5s_6s_2s_3s_4, s_2s_3s_4s_5s_6s_3s_4s_5s_1s_2s_3s_4),\\
 (s_5s_6s_2s_3s_4, s_2s_3s_4s_5s_6s_3s_4s_5s_2s_3s_4),\\
 (s_6s_2s_3s_4s_5s_3s_1, s_2s_3s_4s_5s_6s_3s_4s_5s_4s_1s_2s_3s_1),\\
 (s_6s_1s_2, s_4s_5s_6s_4s_5s_4s_1s_2s_1),\\
 (s_6s_3s_4s_3, s_4s_5s_6s_3s_4s_5s_2s_3s_4s_3),\\
 (s_5s_3s_1s_2s_1, s_3s_4s_5s_3s_4s_1s_2s_3s_1s_2s_1),\\
 (s_3s_4s_5s_2s_3s_1, s_3s_4s_5s_6s_2s_3s_4s_5s_1s_2s_3s_4),\\
 (s_6s_5s_4s_3s_1, s_6s_3s_4s_5s_1s_2s_3s_4s_2s_3s_1),\\
 (s_5s_6s_4s_5s_3s_4s_1, s_4s_5s_6s_4s_5s_1s_2s_3s_4s_2s_3s_2s_1),\\
 (s_5s_6s_2s_3s_4s_5s_4, s_4s_5s_6s_2s_3s_4s_5s_2s_3s_4s_3s_1s_2),\\
 (s_3s_4, s_2s_3s_4s_5s_2s_3s_4s_2),\\
 (s_5s_6s_5s_2s_3s_2, s_2s_3s_4s_5s_6s_3s_4s_5s_4s_2s_3s_2),\\
 (s_5s_6s_5s_2s_3s_1s_2, s_5s_6s_2s_3s_4s_5s_3s_4s_1s_2s_3s_1s_2),\\
 (s_4s_5s_3s_4s_1s_2, s_4s_5s_6s_3s_4s_5s_1s_2s_3s_4s_1s_2),\\
 (s_6s_4s_5s_3s_4s_3s_1, s_3s_4s_5s_6s_2s_3s_4s_5s_2s_3s_4s_3s_1),\\
 (s_6s_4s_5s_4s_1s_2s_1, s_4s_5s_6s_1s_2s_3s_4s_5s_4s_3s_1s_2s_1),\\
 (s_4s_5s_3s_4, s_4s_5s_6s_3s_4s_5s_2s_3s_4s_1),\\
 (s_6s_3s_4s_5s_3s_2, s_2s_3s_4s_5s_6s_2s_3s_4s_5s_4s_3s_2),\\
 (s_5s_6s_5s_4s_2s_1, s_3s_4s_5s_6s_2s_3s_4s_5s_3s_4s_2s_1),\\
 (s_4s_5s_6s_4s_5s_4s_3, s_2s_3s_4s_5s_6s_3s_4s_5s_3s_4s_1s_2s_3),\\
 (s_6s_4s_2s_3, s_4s_5s_6s_5s_2s_3s_4s_1s_2s_3),\\
 (s_5s_1s_2s_3s_2, s_5s_6s_1s_2s_3s_4s_5s_2s_3s_1s_2),\\
 (s_6s_3s_4s_5s_3s_1s_2, s_6s_1s_2s_3s_4s_5s_2s_3s_4s_2s_3s_1s_2),\\
 (s_6s_4s_5s_1s_2s_1, s_2s_3s_4s_5s_6s_4s_5s_4s_3s_1s_2s_1),\\
 (s_3s_4s_5s_6s_3s_4, s_1s_2s_3s_4s_5s_6s_3s_4s_5s_2s_3s_1),\\
 (s_4s_5s_6s_4s_5s_4s_2, s_2s_3s_4s_5s_6s_3s_4s_5s_3s_4s_1s_2s_1),\\
 (s_6s_4s_1s_2s_1, s_4s_5s_6s_5s_1s_2s_3s_4s_1s_2s_1),\\
 (s_2s_3s_4s_5s_3s_1s_2, s_2s_3s_4s_5s_6s_1s_2s_3s_4s_5s_4s_2s_3),\\
 (s_5s_6s_5s_1s_2s_3s_2, s_1s_2s_3s_4s_5s_6s_4s_5s_3s_4s_2s_3s_2),\\
 (s_6s_4s_5s_4s_1s_2s_1, s_4s_5s_6s_3s_4s_5s_2s_3s_4s_3s_1s_2s_1),\\
 (s_5s_4s_1, s_4s_5s_6s_4s_5s_2s_3s_4s_1),\\
 (s_5s_6s_4s_5s_2s_3s_1, s_2s_3s_4s_5s_6s_4s_5s_3s_4s_1s_2s_3s_1),\\
 (s_5s_1s_2s_3s_4s_2s_1, s_4s_5s_6s_2s_3s_4s_5s_1s_2s_3s_4s_2s_1),\\
 (s_4s_5s_6s_4s_2s_3s_1, s_2s_3s_4s_5s_6s_2s_3s_4s_1s_2s_3s_2s_1),\\
 (s_6s_4s_5s_3s_1, s_4s_5s_6s_3s_4s_5s_4s_1s_2s_3s_1),\\
 (s_5s_1s_2s_3s_4s_1, s_5s_6s_2s_3s_4s_5s_1s_2s_3s_4s_2s_1),\\
 (s_5s_6s_5s_2s_3s_4s_1, s_2s_3s_4s_5s_6s_3s_4s_5s_1s_2s_3s_4s_2),\\
 (s_5s_6s_5s_4s_1s_2s_1, s_5s_6s_1s_2s_3s_4s_5s_3s_4s_2s_3s_2s_1),\\
 (s_4s_5s_3s_4s_3s_1, s_4s_5s_6s_3s_4s_5s_2s_3s_4s_1s_2s_3),\\
 (s_4s_5s_6s_4s_2s_3, s_1s_2s_3s_4s_5s_6s_3s_4s_1s_2s_3s_1),\\
 (s_4s_5s_6s_4s_5s_4s_2, s_1s_2s_3s_4s_5s_6s_2s_3s_4s_5s_4s_1s_2),\\
 (s_3s_4s_5s_4s_3s_1s_2, s_5s_6s_1s_2s_3s_4s_5s_3s_4s_2s_3s_1s_2),\\
 (s_5s_6s_5s_3s_1, s_1s_2s_3s_4s_5s_6s_5s_4s_3s_2s_1),\\
 (s_5s_4s_1s_2s_3, s_5s_6s_3s_4s_5s_2s_3s_4s_1s_2s_3),\\
 (s_4s_5s_4s_1s_2s_1, s_4s_5s_6s_3s_4s_5s_3s_4s_2s_3s_1s_2),\\
 (s_5s_3s_4s_3, s_5s_6s_4s_5s_3s_4s_1s_2s_3s_1),\\
 (s_4s_5s_6s_5s_2s_3, s_2s_3s_4s_5s_6s_3s_4s_5s_4s_1s_2s_3),\\
 (s_6s_3s_4s_2, s_3s_4s_5s_6s_4s_5s_3s_4s_3s_2),\\
 (s_6s_3s_4s_3s_1, s_3s_4s_5s_6s_5s_2s_3s_4s_2s_3s_1),\\
 (s_5s_6s_2s_3s_4s_5s_4s_1s_2s_1,
  s_2s_3s_4s_5s_6s_1s_2s_3s_4s_5s_3s_4s_3s_1s_2s_1),\\
 (s_5s_6s_5s_1s_2s_3s_1s_2s_1, s_1s_2s_3s_4s_5s_6s_4s_5s_4s_1s_2s_3s_1s_2s_1),\\
 (s_5s_6s_5s_2s_3s_4s_1s_2s_1, s_2s_3s_4s_5s_6s_4s_5s_1s_2s_3s_4s_3s_1s_2s_1)$
\addcontentsline{toc}{section}{Appendix B}
\section*{Appendix B}
\subsection*{Length 4}
17 distinct intervals

 $(s_5s_4s_5s_4s_1s_2s_3s_1s_2, s_4s_5s_1s_2s_3s_4s_5s_3s_4s_2s_3s_1s_2),\\
 (s_5s_3s_4s_5s_4s_1s_2s_3s_2s_1, s_5s_3s_4s_5s_1s_2s_3s_4s_5s_3s_4s_2s_3s_1),\\
 (s_5s_1s_2s_3s_4s_5s_3s_2, s_5s_1s_2s_3s_4s_5s_3s_4s_2s_3s_1s_2),\\
 (s_5s_4s_5s_2s_3s_4s_3s_1s_2, s_5s_4s_5s_2s_3s_4s_5s_1s_2s_3s_4s_5s_2),\\
 (s_5s_2s_3s_4s_5s_2s_3s_2, s_5s_3s_4s_5s_2s_3s_4s_5s_4s_2s_3s_2),\\
 (s_5s_2s_3s_4s_5s_3s_4, s_5s_4s_5s_2s_3s_4s_5s_4s_2s_3s_2),\\
 (s_4s_5s_2s_3s_4s_5s_2s_3, s_4s_5s_3s_4s_5s_2s_3s_4s_5s_1s_2s_3),\\
 (s_4s_5s_3s_4s_5s_4s_1s_2, s_3s_4s_5s_2s_3s_4s_5s_2s_3s_4s_1s_2),\\
 (s_4s_5s_3s_4s_5s_4s_1s_2, s_3s_4s_5s_1s_2s_3s_4s_5s_2s_3s_4s_2),\\
 (s_5s_3s_4s_5s_4s_3s_1, s_5s_1s_2s_3s_4s_5s_2s_3s_4s_3s_2),\\
 (s_4s_5s_2s_3s_4s_3s_1s_2, s_4s_5s_3s_4s_5s_2s_3s_4s_2s_3s_1s_2),\\
 (s_5s_2s_3s_4s_5s_4s_1s_2s_3, s_2s_3s_4s_5s_1s_2s_3s_4s_5s_3s_4s_2s_3),\\
 (s_5s_4s_5s_2s_3s_4s_5s_4, s_5s_4s_5s_2s_3s_4s_5s_4s_2s_3s_1s_2),\\
 (s_5s_4s_5s_3s_4s_5s_3s_4s_3s_1, s_4s_5s_3s_4s_5s_1s_2s_3s_4s_5s_2s_3s_4s_1),\\
 (s_5s_4s_5s_2s_3s_1s_2s_1, s_5s_3s_4s_5s_2s_3s_4s_5s_2s_3s_1s_2),\\
 (s_3s_4s_5s_2s_3s_4s_5s_3s_4, s_5s_4s_5s_3s_4s_5s_2s_3s_4s_5s_2s_3s_4),\\
 (s_5s_2s_3s_4s_5s_2s_3s_2, s_5s_2s_3s_4s_5s_1s_2s_3s_4s_5s_2s_3)$
\subsection*{Length 5}
143 distinct intervals

$(s_4s_5s_3s_4s_5s_1s_2s_3s_4, s_4s_5s_3s_4s_5s_1s_2s_3s_4s_5s_1s_2s_3s_1),\\
 (s_5s_4s_5s_2s_3s_4s_3s_1, s_4s_5s_2s_3s_4s_5s_1s_2s_3s_4s_2s_3s_1),\\
 (s_4s_5s_2s_3s_4s_5s_1s_2s_3s_2, s_2s_3s_4s_5s_1s_2s_3s_4s_5s_2s_3s_4s_1s_2s_3),\\
 (s_5s_2s_3s_4s_5s_2s_3s_2, s_5s_3s_4s_5s_2s_3s_4s_5s_3s_4s_2s_3s_2),\\
 (s_5s_3s_4s_5s_4s_1s_2s_1, s_3s_4s_5s_1s_2s_3s_4s_5s_3s_4s_1s_2s_1),\\
 (s_3s_4s_5s_2s_3s_2, s_3s_4s_5s_2s_3s_4s_5s_1s_2s_3s_2),\\
 (s_4s_5s_2s_3s_4, s_4s_5s_2s_3s_4s_5s_1s_2s_3s_4),\\
 (s_5s_4s_5s_1s_2s_3s_4s_3s_2s_1, s_4s_5s_3s_4s_5s_1s_2s_3s_4s_5s_2s_3s_4s_3s_1),\\
 (s_4s_5s_4s_3s_2, s_4s_5s_2s_3s_4s_5s_3s_4s_2s_3),\\
 (s_4s_5s_2s_3s_4s_5s_4s_1s_2s_1, s_2s_3s_4s_5s_1s_2s_3s_4s_5s_2s_3s_4s_1s_2s_3),\\
 (s_5s_4s_1s_2s_3s_1s_2s_1, s_5s_4s_5s_2s_3s_4s_5s_4s_1s_2s_3s_1s_2),\\
 (s_3s_4s_5s_2s_3s_2, s_4s_5s_3s_4s_5s_2s_3s_4s_5s_2s_3),\\
 (s_1s_2s_3s_4s_5s_1s_2s_3s_2, s_4s_5s_2s_3s_4s_5s_1s_2s_3s_4s_5s_4s_3s_2),\\
 (s_1s_2s_3s_4s_5s_1s_2s_3s_2s_1, s_5s_2s_3s_4s_5s_1s_2s_3s_4s_5s_4s_2s_3s_2s_1),\\
 (s_4s_5s_2s_3s_4s_5s_2s_3, s_4s_5s_3s_4s_5s_2s_3s_4s_5s_4s_1s_2s_3),\\
 (s_5s_3s_4s_5s_1s_2s_3s_2s_1, s_5s_3s_4s_5s_1s_2s_3s_4s_5s_4s_2s_3s_2s_1),\\
 (s_5s_2s_3s_4s_5s_2s_3s_1s_2s_1, s_5s_4s_5s_3s_4s_5s_2s_3s_4s_5s_4s_3s_1s_2s_1),\\
 (s_4s_5s_2s_3s_4s_5s_3s_4s_3s_1, s_4s_5s_3s_4s_5s_2s_3s_4s_5s_2s_3s_4s_2s_3s_1),\\
 (s_1s_2s_3s_4s_1s_2s_3s_1s_2s_1, s_4s_5s_2s_3s_4s_5s_1s_2s_3s_4s_5s_4s_3s_1s_2),\\
 (s_5s_3s_4s_5s_4s_3s_1s_2s_1, s_5s_3s_4s_5s_2s_3s_4s_5s_3s_4s_2s_3s_1s_2),\\
 (s_5s_4s_5s_4s_2, s_4s_5s_2s_3s_4s_5s_4s_3s_1s_2),\\
 (s_5s_4s_5s_4s_1s_2s_3s_1s_2, s_5s_4s_5s_1s_2s_3s_4s_5s_4s_1s_2s_3s_1s_2),\\
 (s_4s_5s_2s_3s_4s_5s_2s_3s_4, s_4s_5s_3s_4s_5s_2s_3s_4s_5s_1s_2s_3s_4s_1),\\
 (s_4s_5s_3s_4, s_4s_5s_3s_4s_5s_2s_3s_4s_3),\\
 (s_5s_3s_4s_5s_4s_1s_2s_3s_1, s_3s_4s_5s_1s_2s_3s_4s_5s_1s_2s_3s_4s_2s_3),\\
 (s_5s_2s_3s_4s_2s_3s_2, s_5s_3s_4s_5s_2s_3s_4s_5s_1s_2s_3s_4),\\
 (s_5s_4s_5s_3s_4s_5s_4s_1s_2s_1, s_5s_3s_4s_5s_2s_3s_4s_5s_1s_2s_3s_4s_5s_1s_2),\\
 (s_5s_4s_5s_2s_3s_4s_5s_3s_4s_3, s_4s_5s_2s_3s_4s_5s_1s_2s_3s_4s_5s_2s_3s_4s_1),\\
 (s_5s_4s_5s_4s_1s_2s_1, s_4s_5s_1s_2s_3s_4s_5s_4s_2s_3s_1s_2),\\
 (s_4s_5s_1s_2s_3s_4s_5s_3s_4s_3, s_4s_5s_3s_4s_5s_1s_2s_3s_4s_5s_1s_2s_3s_4s_3),\\
 (s_1s_2s_3s_4s_5s_4s_2s_1, s_4s_5s_3s_4s_5s_1s_2s_3s_4s_5s_3s_4s_1),\\
 (s_5s_3s_4s_5s_4s_1s_2s_3s_2, s_3s_4s_5s_2s_3s_4s_5s_3s_4s_1s_2s_3s_1s_2),\\
 (s_5s_1s_2s_3s_4s_2s_3s_1, s_5s_3s_4s_5s_2s_3s_4s_5s_1s_2s_3s_4s_1),\\
 (s_5s_2s_3s_4s_5s_1s_2s_3s_1s_2, s_5s_2s_3s_4s_5s_1s_2s_3s_4s_5s_3s_4s_1s_2s_3),\\
 (s_5s_2s_3s_4s_5s_4s_1s_2s_3s_1, s_5s_2s_3s_4s_5s_1s_2s_3s_4s_5s_3s_4s_1s_2s_3),\\
 (s_5s_4s_5s_4s_1s_2s_1, s_4s_5s_1s_2s_3s_4s_5s_3s_4s_1s_2s_1),\\
 (s_4s_5s_2s_3s_4s_5s_2s_3s_2s_1, s_4s_5s_3s_4s_5s_2s_3s_4s_5s_4s_2s_3s_1s_2s_1),\\
 (s_4s_5s_1s_2s_3s_4s_5s_3s_1, s_4s_5s_2s_3s_4s_5s_1s_2s_3s_4s_5s_4s_2s_3),\\
 (s_5s_2s_3s_4s_5s_3s_4s_2, s_5s_4s_5s_2s_3s_4s_5s_1s_2s_3s_4s_1s_2),\\
 (s_5s_3s_4s_5s_1s_2s_3s_1, s_4s_5s_3s_4s_5s_2s_3s_4s_5s_1s_2s_3s_1),\\
 (s_4s_5s_1s_2s_3s_4s_1s_2s_3, s_4s_5s_1s_2s_3s_4s_5s_1s_2s_3s_4s_1s_2s_3),\\
 (s_3s_4s_5s_2s_3s_2s_1, s_3s_4s_5s_2s_3s_4s_5s_1s_2s_3s_4s_1),\\
 (s_5s_4s_5s_3s_4s_5s_3s_4s_1s_2, s_4s_5s_3s_4s_5s_1s_2s_3s_4s_5s_2s_3s_4s_1s_2),\\
 (s_3s_4s_5s_1s_2s_3s_4s_2s_3s_2, s_3s_4s_5s_1s_2s_3s_4s_5s_1s_2s_3s_4s_2s_3s_2),\\
 (s_4s_5s_3s_4s_5s_4s_1s_2, s_4s_5s_3s_4s_5s_1s_2s_3s_4s_5s_2s_3s_4),\\
 (s_5s_3s_4s_5s_4s_1s_2s_3s_2, s_3s_4s_5s_1s_2s_3s_4s_5s_3s_4s_2s_3s_1s_2),\\
 (s_3s_4s_5s_2s_3s_4, s_3s_4s_5s_2s_3s_4s_5s_1s_2s_3s_4),\\
 (s_4s_5s_3s_1, s_4s_5s_1s_2s_3s_4s_2s_3s_2),\\
 (s_4s_5s_2s_3s_4s_3s_1s_2, s_4s_5s_2s_3s_4s_5s_2s_3s_4s_2s_3s_1s_2),\\
 (s_5s_4s_5s_4s_1s_2s_1, s_4s_5s_1s_2s_3s_4s_5s_2s_3s_4s_1s_2),\\
 (s_5s_3s_4s_5s_2s_3s_4s_2s_3s_1, s_5s_3s_4s_5s_2s_3s_4s_5s_1s_2s_3s_4s_1s_2s_3),\\
 (s_2s_3s_4s_2s_3s_1s_2s_1, s_4s_5s_2s_3s_4s_5s_3s_4s_2s_3s_1s_2s_1),\\
 (s_5s_4s_5s_4s_1s_2s_3s_2, s_5s_4s_5s_1s_2s_3s_4s_5s_3s_4s_2s_3s_2),\\
 (s_3s_4s_5s_1s_2s_3s_4s_5s_2s_3, s_3s_4s_5s_2s_3s_4s_5s_1s_2s_3s_4s_5s_4s_2s_3),\\
 (s_5s_2s_3s_4s_5s_4s_2s_1, s_2s_3s_4s_5s_1s_2s_3s_4s_5s_3s_4s_2s_1),\\
 (s_5s_1s_2s_3s_4s_5s_1s_2s_3s_1, s_3s_4s_5s_2s_3s_4s_5s_1s_2s_3s_4s_5s_4s_2s_3),\\
 (s_5s_1s_2s_3s_4s_2s_3s_2, s_5s_3s_4s_5s_1s_2s_3s_4s_1s_2s_3s_1s_2),\\
 (s_5s_3s_4s_5s_2s_3s_2s_1, s_5s_3s_4s_5s_2s_3s_4s_5s_1s_2s_3s_2s_1),\\
 (s_5s_3s_4s_5s_4s_2s_3s_1, s_3s_4s_5s_2s_3s_4s_5s_3s_4s_1s_2s_3s_1),\\
 (s_5s_4s_5s_1s_2s_3s_4s_3s_1s_2, s_5s_4s_5s_1s_2s_3s_4s_5s_1s_2s_3s_4s_3s_1s_2),\\
 (s_5s_2s_3s_4s_5s_3s_2, s_5s_4s_5s_2s_3s_4s_5s_4s_2s_3s_1s_2),\\
 (s_3s_4s_5s_2s_3s_4s_5s_1s_2s_3, s_3s_4s_5s_2s_3s_4s_5s_1s_2s_3s_4s_5s_4s_2s_3),\\
 (s_5s_3s_4s_5s_4s_1s_2s_1, s_5s_3s_4s_5s_2s_3s_4s_5s_3s_4s_1s_2s_1),\\
 (s_5s_4s_5s_4s_1s_2s_3s_1, s_4s_5s_1s_2s_3s_4s_5s_2s_3s_4s_1s_2s_3),\\
 (s_5s_1s_2s_3s_4s_2s_3s_2, s_4s_5s_3s_4s_5s_1s_2s_3s_4s_1s_2s_3s_1),\\
 (s_5s_3s_4s_5s_4s_1s_2s_3s_1, s_3s_4s_5s_1s_2s_3s_4s_5s_3s_4s_1s_2s_3s_1),\\
 (s_1s_2s_3s_4s_1s_2, s_5s_1s_2s_3s_4s_5s_3s_4s_3s_1s_2),\\
 (s_5s_2s_3s_4s_5s_1s_2s_1, s_5s_2s_3s_4s_5s_1s_2s_3s_4s_5s_3s_1s_2),\\
 (s_4s_5s_3s_4s_5s_2s_3s_4s_2s_3, s_4s_5s_3s_4s_5s_2s_3s_4s_5s_2s_3s_4s_1s_2s_3),\\
 (s_5s_1s_2s_3s_4s_1s_2s_3, s_2s_3s_4s_5s_1s_2s_3s_4s_5s_3s_4s_2s_3),\\
 (s_5s_1s_2s_3s_1s_2, s_5s_1s_2s_3s_4s_5s_4s_2s_3s_1s_2),\\
 (s_5s_4s_1s_2s_3, s_5s_1s_2s_3s_4s_5s_3s_4s_2s_3),\\
 (s_4s_5s_2s_3s_4s_2s_1, s_4s_5s_2s_3s_4s_5s_4s_2s_3s_1s_2s_1),\\
 (s_4s_5s_3s_4s_5s_2s_3s_4s_2s_1, s_5s_4s_5s_3s_4s_5s_2s_3s_4s_5s_1s_2s_3s_4s_1),\\
 (s_5s_2s_3s_4s_5s_2s_3s_2, s_5s_3s_4s_5s_2s_3s_4s_5s_4s_1s_2s_3s_2),\\
 (s_4s_5s_3s_4s_5s_2s_3s_4s_1, s_4s_5s_3s_4s_5s_2s_3s_4s_5s_1s_2s_3s_4s_1),\\
 (s_1s_2s_3s_4s_5s_2s_3s_2s_1, s_3s_4s_5s_2s_3s_4s_5s_1s_2s_3s_4s_5s_3s_1),\\
 (s_5s_3s_4s_5s_2s_3s_4s_3s_1, s_5s_3s_4s_5s_2s_3s_4s_5s_1s_2s_3s_4s_2s_1),\\
 (s_5s_3s_4s_5s_3s_4s_2s_3s_1, s_5s_3s_4s_5s_2s_3s_4s_5s_2s_3s_4s_1s_2s_3),\\
 (s_5s_2s_3s_4s_5s_2s_3s_4, s_2s_3s_4s_5s_1s_2s_3s_4s_5s_1s_2s_3s_4),\\
 (s_5s_4s_5s_2s_3s_4s_5s_3s_2, s_4s_5s_2s_3s_4s_5s_1s_2s_3s_4s_5s_3s_4s_2),\\
 (s_4s_5s_2s_3s_4s_5s_3s_4s_3s_1, s_4s_5s_2s_3s_4s_5s_1s_2s_3s_4s_5s_2s_3s_4s_1),\\
 (s_5s_4s_5s_3s_4s_5s_3s_4s_3s_1, s_4s_5s_2s_3s_4s_5s_1s_2s_3s_4s_5s_2s_3s_4s_1),\\
 (s_5s_4s_5s_3s_1s_2s_1, s_5s_1s_2s_3s_4s_5s_3s_4s_2s_3s_1s_2),\\
 (s_3s_4s_5s_4s_1s_2s_1, s_5s_3s_4s_5s_1s_2s_3s_4s_5s_3s_4s_1),\\
 (s_3s_4s_5s_2s_3s_4s_2s_1, s_3s_4s_5s_2s_3s_4s_5s_1s_2s_3s_4s_2s_1),\\
 (s_5s_1s_2s_3s_4s_1s_2s_3s_2, s_5s_4s_5s_2s_3s_4s_5s_1s_2s_3s_4s_1s_2s_1),\\
 (s_4s_5s_3s_4s_5s_2s_3s_4, s_4s_5s_3s_4s_5s_2s_3s_4s_5s_1s_2s_3s_4),\\
 (s_5s_3s_4s_5s_2s_3s_4s_3, s_3s_4s_5s_2s_3s_4s_5s_1s_2s_3s_4s_2s_3),\\
 (s_5s_4s_5s_1s_2s_3s_4s_5s_3s_1, s_3s_4s_5s_2s_3s_4s_5s_1s_2s_3s_4s_5s_2s_3s_1),\\
 (s_1s_2s_3s_4s_5s_1s_2s_3s_2s_1, s_4s_5s_2s_3s_4s_5s_1s_2s_3s_4s_5s_4s_1s_2s_3),\\
 (s_1s_2s_3s_4s_1s_2, s_5s_4s_5s_1s_2s_3s_4s_5s_4s_1s_2),\\
 (s_5s_1s_2s_3s_4s_5s_2s_3s_2, s_5s_3s_4s_5s_1s_2s_3s_4s_5s_2s_3s_4s_3s_2),\\
 (s_5s_3s_4s_5s_4s_1s_2s_1, s_2s_3s_4s_5s_1s_2s_3s_4s_5s_3s_4s_2s_1),\\
 (s_4s_5s_3s_4s_1, s_4s_5s_3s_4s_5s_1s_2s_3s_4s_1),\\
 (s_5s_3s_4s_5s_4s_1s_2s_3s_2s_1, s_5s_1s_2s_3s_4s_5s_1s_2s_3s_4s_1s_2s_3s_2s_1),\\
 (s_5s_4s_5s_4s_2s_3s_1, s_4s_5s_2s_3s_4s_5s_3s_4s_1s_2s_3s_1),\\
 (s_4s_5s_4s_2s_3, s_4s_5s_2s_3s_4s_5s_3s_4s_2s_3),\\
 (s_5s_2s_3s_4s_5s_4s_3s_1s_2, s_5s_4s_5s_2s_3s_4s_5s_1s_2s_3s_4s_3s_1s_2),\\
 (s_4s_5s_2s_3s_4s_2, s_4s_5s_3s_4s_5s_2s_3s_4s_3s_1s_2),\\
 (s_5s_4s_5s_1s_2s_3s_4s_2s_3s_1, s_5s_4s_5s_2s_3s_4s_5s_1s_2s_3s_4s_1s_2s_3s_2),\\
 (s_5s_4s_5s_2s_3s_4s_3s_1s_2, s_4s_5s_2s_3s_4s_5s_1s_2s_3s_4s_5s_2s_3s_4),\\
 (s_5s_3s_4s_5s_4s_1s_2s_3s_1s_2, s_3s_4s_5s_2s_3s_4s_5s_1s_2s_3s_4s_2s_3s_1s_2),\\
 (s_4s_5s_1s_2s_3, s_3s_4s_5s_1s_2s_3s_4s_5s_2s_3),\\
 (s_5s_1s_2s_3s_4s_5s_2s_1, s_5s_3s_4s_5s_1s_2s_3s_4s_5s_4s_3s_2s_1),\\
 (s_5s_3s_4s_5s_4s_2s_3, s_3s_4s_5s_2s_3s_4s_5s_3s_4s_1s_2s_3),\\
 (s_5s_1s_2s_3s_4s_5s_2s_3s_2, s_5s_3s_4s_5s_1s_2s_3s_4s_5s_4s_2s_3s_1s_2),\\
 (s_5s_2s_3s_4s_5s_1s_2s_3s_4s_1, s_4s_5s_3s_4s_5s_2s_3s_4s_5s_1s_2s_3s_4s_1s_2),\\
 (s_5s_1s_2s_3s_4s_1s_2s_3s_1, s_5s_2s_3s_4s_5s_1s_2s_3s_4s_1s_2s_3s_2s_1),\\
 (s_5s_3s_4s_5s_4s_2, s_4s_5s_3s_4s_5s_2s_3s_4s_5s_4s_2),\\
 (s_4s_5s_1s_2s_3s_1s_2s_1, s_4s_5s_1s_2s_3s_4s_5s_4s_1s_2s_3s_1s_2),\\
 (s_5s_3s_4s_5s_3s_4s_3s_1s_2s_1, s_4s_5s_1s_2s_3s_4s_5s_2s_3s_4s_2s_3s_1s_2s_1),\\
 (s_5s_1s_2s_3s_4s_5s_4s_1s_2s_1, s_4s_5s_2s_3s_4s_5s_1s_2s_3s_4s_5s_4s_3s_2s_1),\\
 (s_5s_3s_4s_5s_4s_2s_3s_1s_2, s_3s_4s_5s_2s_3s_4s_5s_2s_3s_4s_2s_3s_1s_2),\\
 (s_5s_4s_5s_3s_4s_5s_3s_4s_1s_2, s_5s_3s_4s_5s_2s_3s_4s_5s_3s_4s_1s_2s_3s_1s_2),\\
 (s_5s_4s_5s_3s_4s_5s_3s_4s_2s_1, s_4s_5s_3s_4s_5s_2s_3s_4s_5s_1s_2s_3s_4s_1s_2),\\
 (s_5s_2s_3s_4s_5s_4s_1s_2s_3s_1, s_2s_3s_4s_5s_1s_2s_3s_4s_5s_3s_4s_1s_2s_3s_1),\\
 (s_5s_2s_3s_4s_5s_2s_3s_2, s_5s_3s_4s_5s_2s_3s_4s_5s_4s_2s_3s_2s_1),\\
 (s_4s_5s_2s_3s_4s_1s_2, s_4s_5s_2s_3s_4s_5s_1s_2s_3s_4s_1s_2),\\
 (s_5s_4s_5s_4s_1s_2s_3s_1s_2, s_4s_5s_1s_2s_3s_4s_5s_3s_4s_1s_2s_3s_1s_2),\\
 (s_4s_5s_3s_4s_5s_3s_4s_1, s_4s_5s_3s_4s_5s_1s_2s_3s_4s_5s_2s_3s_4),\\
 (s_5s_2s_3s_4s_5s_1s_2, s_2s_3s_4s_5s_1s_2s_3s_4s_5s_3s_4s_2),\\
 (s_5s_2s_3s_4s_5s_3s_1s_2s_1, s_5s_3s_4s_5s_2s_3s_4s_5s_4s_2s_3s_1s_2s_1),\\
 (s_4s_5s_2s_3s_4s_5s_1s_2, s_4s_5s_2s_3s_4s_5s_1s_2s_3s_4s_5s_1s_2),\\
 (s_5s_2s_3s_4s_5s_2, s_4s_5s_3s_4s_5s_2s_3s_4s_5s_3s_2),\\
 (s_5s_4s_5s_4s_1s_2, s_5s_4s_5s_1s_2s_3s_4s_5s_3s_4s_2),\\
 (s_4s_5s_1s_2s_3s_4s_5s_4s_2s_3, s_4s_5s_2s_3s_4s_5s_1s_2s_3s_4s_5s_4s_1s_2s_3),\\
 (s_3s_4s_5s_1s_2s_3s_2, s_5s_4s_5s_3s_4s_5s_4s_1s_2s_3s_1s_2),\\
 (s_5s_3s_4s_5s_2s_3s_4s_2s_3s_2, s_5s_4s_5s_3s_4s_5s_2s_3s_4s_5s_1s_2s_3s_4s_2),\\
 (s_4s_5s_2s_3s_4s_5s_2s_3, s_4s_5s_3s_4s_5s_2s_3s_4s_5s_1s_2s_3s_4),\\
 (s_5s_3s_4s_2s_3, s_5s_3s_4s_5s_2s_3s_4s_1s_2s_3),\\
 (s_4s_5s_2s_3s_4s_5s_2s_3s_4s_2, s_4s_5s_3s_4s_5s_2s_3s_4s_5s_1s_2s_3s_4s_3s_2),\\
 (s_4s_5s_2s_3s_4s_5s_3s_4s_3s_1, s_3s_4s_5s_2s_3s_4s_5s_1s_2s_3s_4s_5s_2s_3s_4),\\
 (s_5s_4s_5s_2s_3s_4s_5s_3s_4s_1s_2s_3,
  s_4s_5s_2s_3s_4s_5s_1s_2s_3s_4s_5s_2s_3s_4s_1s_2s_3),\\
 (s_5s_1s_2s_3s_4s_5s_2s_3s_4s_1s_2s_1,
  s_5s_3s_4s_5s_1s_2s_3s_4s_5s_1s_2s_3s_4s_3s_1s_2s_1),\\
 (s_5s_3s_4s_5s_2s_3s_4s_5s_3s_4s_1s_2s_3,
  s_3s_4s_5s_2s_3s_4s_5s_1s_2s_3s_4s_5s_2s_3s_4s_1s_2s_3),\\
 (s_5s_4s_5s_1s_2s_3s_4s_5s_1s_2s_3s_1s_2s_1,
  s_5s_4s_5s_2s_3s_4s_5s_1s_2s_3s_4s_5s_4s_1s_2s_3s_1s_2s_1),\\
 (s_5s_4s_5s_2s_3s_4s_5s_3s_4s_1s_2,
  s_4s_5s_2s_3s_4s_5s_1s_2s_3s_4s_5s_2s_3s_4s_1s_2),\\
 (s_5s_1s_2s_3s_4s_5s_1s_2s_3s_1s_2,
  s_5s_2s_3s_4s_5s_1s_2s_3s_4s_5s_4s_1s_2s_3s_1s_2),\\
 (s_5s_4s_5s_1s_2s_3s_4s_5s_1s_2s_3s_2,
  s_5s_4s_5s_2s_3s_4s_5s_1s_2s_3s_4s_5s_4s_1s_2s_3s_2),\\
 (s_5s_1s_2s_3s_4s_5s_1s_2s_3s_1s_2s_1,
  s_5s_2s_3s_4s_5s_1s_2s_3s_4s_5s_4s_1s_2s_3s_1s_2s_1),\\
 (s_5s_2s_3s_4s_5s_1s_2s_3s_4s_5s_1s_2s_3s_1s_2,
  s_5s_3s_4s_5s_2s_3s_4s_5s_1s_2s_3s_4s_5s_3s_4s_1s_2s_3s_1s_2),\\
 (s_5s_3s_4s_5s_2s_3s_4s_5s_3s_4s_1,
  s_3s_4s_5s_2s_3s_4s_5s_1s_2s_3s_4s_5s_2s_3s_4s_1)$
\end{document}